\documentclass[a4paper,11pt]{amsart}

\usepackage[english]{babel}
\usepackage[utf8]{inputenc}
\usepackage{amsmath, amssymb, amscd, amsthm, amsfonts}
\usepackage{graphicx,tikz}
\usepackage{hyperref}
\usepackage{geometry}
\usepackage{textcomp}

\theoremstyle{plain}
\newtheorem{theorem}{Theorem}
\newtheorem{mainthm}{Theorem}

\newtheorem{lemma}[theorem]{Lemma}

\newtheorem{prop}[theorem]{Proposition}

\theoremstyle{definition}
\newtheorem{defn}{Definition}

\DeclareMathOperator{\Aut}{Aut}
\DeclareMathOperator{\St}{St}
\DeclareMathOperator{\Sym}{Sym}

\DeclareUnicodeCharacter{0308}{}

\bibliography{bibliography.bib}

\begin{document}

\title{MULTI-EGS-GROUPS: EXPONENT OF CONGRUENCE QUOTIENTS}

\author{Elena Maini}
\address{Department of Mathematics, University of Padova, 35121 Padova, Italy}
\email{elenamaini.06@gmail.com}

\maketitle

\begin{abstract}
    Given a multi-EGS-group $K$ acting on the $p$-adic rooted tree, where $p$ is any prime number, we compute the exponent of the congruence quotient $K_n= K/ \St_K(n)$ for all $n\ge 1$. The formula that we obtain for $\exp(K_n)$ only depends on $p$, $n$ and the periodicity of $K$.
\end{abstract}

\section{Introduction}
From the early 20th century, the so-called Burnside Problem has been an object of great interest, being formulated in different versions which led to various and elaborate solutions.
One of its versions,
known as the General Burnside Problem,
asks
whether a finitely generated periodic group has to be finite. Since the 1960s, several examples of infinite finitely generated periodic groups have been found: among them, some of the easiest turned out to be groups acting on regular rooted trees. As an evidence of this, we mention the Grigorchuk groups (defined in \cite{Grigorchuk}) and the Gupta-Sidki group (defined in \cite{GuptaSidki}).
The Gupta-Sidki group belongs to a broader family of subgroups of the group of automorphisms of the $p$-adic rooted tree $\mathcal{T}$ (where $p$ is a prime number), which will be the object of our paper: the family of the so-called multi-EGS-groups.

In order to define this family of groups, we start by recalling some basic notions from the theory of automorphisms of trees. We use the following notation: the vertices of $\mathcal{T}$ are the elements of the free monoid generated by $X=\{1,\dots,p\}$, that is, the words of finite length in the alphabet $X$ (the word of length zero, denoted by $\emptyset$, is said to be the root of $\mathcal{T}$ and is the only vertex that has no parent).
For $n\in \mathbb{N}$, we denote with $X^n$ the set of all the words of length $n$; the finite tree whose vertices are the words of length $\leq n$ is $\mathcal{T}_n$ and $\Aut\mathcal{T}_n \simeq \Aut\mathcal{T}/\St(n)$, where $\St(n)$ is the subgroup of $\Aut\mathcal{T}$ consisting of all the automorphisms that do not move the elements of $X^n$.
If $G$ is a subgroup of $\Aut\mathcal{T}$, we put $\St_{G}(n) = \St(n) \cap G$ and we call $G_n = G/\St_{G}(n)$ the \textit{nth congruence quotient} of $G$.
The sequence $\{ \St_{G}(n) \}_{n\ge 0}$ is a chain of finite-index normal subgroups of $G$, which makes $G$ into a residually finite group and gives rise to the sequence of finite groups $\{G_n\}_{n\ge 0}$.
Understanding properties of $G_n$ as $n$ varies is a typical way to understand the structure of the whole $G$. One can easily prove, for example, that
\begin{equation}
\label{EQlimiteG_n}
    \exp(G)= \lim_{n\rightarrow \infty} \exp(G_n).
\end{equation}
It is clear that, if $H\le G\le \Aut\mathcal{T}$, $H_n$ can be naturally embedded in $G_n$.
We will write $\St_{G_n}(k)$ for the quotient $\St_{G}(k)/\St_{G}(n)$, whenever $k$ is an integer less than or equal to $n$.

Now, for any $g\in \Aut\mathcal{T}$ and for any vertex $u$, we denote with $g_{(u)}$ the permutation of $X$ which satisfies $g(ux)=g(u)g_{(u)}(x)$ for any $x\in X$. We refer to $g_{(u)}$ as the \textit{label} of $g$ at $u$. The set of all labels of $g$ is the \textit{portrait} of $g$ and the automorphism $g$ is completely determined by its portrait. Moreover, for any vertex $u$, we write $g_u$ for the automorphism of $\Aut\mathcal{T}$ that has portrait given by ${(g_u)}_{(v)} = g_{(uv)}$ for any vertex $v$.
We call $g_u$ the \textit{section} of $g$ at $u$ and we say that a subgroup $G$ of $\Aut\mathcal{T}$ is \textit{self-similar} if, for every $g\in G$, all the sections of $g$ belong to $G$.
It is well known that
the map
\begin{align*}
    \psi\text{ }:\text{  } & \St(1)\text{ }\longrightarrow \text{ } \Aut\mathcal{T} \times \overset{p}{\cdots} \times \Aut\mathcal{T}
    \\
    & \hspace{0,5 cm} g \hspace{0,6 cm} \longmapsto \hspace{0,7 cm} (g_1,\dots,g_p)
\end{align*}
is a group isomorphism and clearly, if $G$ is self-similar, $\psi(\St_{G}(1)) \subseteq G \times \overset{p}{\cdots} \times G$.
From now onwards, we will denote by the letter $a$ the automorphism of $\mathcal{T}$ with portrait 
\begin{equation*}
    a_{(u)}=
    \begin{cases}
    (1\hspace{0,1 cm} 2 \hspace{0,1 cm} \dots \hspace{0,1 cm} p) \hspace{1 cm} & \text{ if } u \text{ is the root of } \mathcal{T} \\
    1 \hspace{1 cm} & \text{ otherwise.}
    \end{cases}
\end{equation*}
Notice that $a$ has order $p$ and then we will be allowed to write $a^e$ whenever $e$ is an element in $\mathbb{F}_p$.
It follows from an easy computation that for any $g\in \St(1)$ with $\psi(g) = (g_1,\dots,g_{p-1},g_p)$ we have
\begin{equation}
\label{EQconjugatebya}
    \psi(g^{a}) = (g_p,g_1,\dots,g_{p-1})
\end{equation}
and then
\begin{align}
\label{EQpth_powers}
\begin{split}
    \psi((a g)^p)&= \psi(g^{(a^{p-1})}) \psi(g^{(a^{p-2})}) \dots \psi(g^{(a^{})}) \psi(g)
    \\
    &= (\hspace{0,1 cm}(g_2 \dots g_{p-1} g_p g_1)\hspace{0,1 cm}, \hspace{0,1 cm}(g_3 \dots g_p g_1 g_2) \hspace{0,1 cm},\hspace{0,1 cm} \dots\hspace{0,1 cm}, \hspace{0,1 cm} (g_1 \dots g_{p-2} g_{p-1} g_{p})\hspace{0,1 cm}).
    \end{split}
\end{align}

\vspace{5 pt}

The isomorphism $\psi$, besides being typically used to understand the structure of groups acting on $\mathcal{T}$, is a useful tool to define recursively automorphisms. As an evidence of this, we give the following definition.
For any vector $\mathbf{e}=(e_1, \dots, e_{p-1})\in (\mathbb{F}_p)^{p-1}$ and for any $j\in X=\{1,\dots ,p\}$, we refer to the automorphism $b\in \St(1)$ defined by 
\begin{equation*}
    \psi(b) = (a^{e_{p-j+1}}, \dots , a^{e_{p-1}}, b, a^{e_1}, \dots, a^{e_{p-j}} )
\end{equation*}
as the \textit{directed automorphism along the path $P_j=(\emptyset, j, jj, \dots)$ with defining vector $\mathbf{e}$}. The use of this term is justified by the fact that the label of $b$ at a vertex $u$ can be non trivial only when the parent of $u$ lies in $P_j$.
We are then ready to say what we mean with the term multi-EGS-group.
\begin{defn}
\label{D1}
    Let $E^{(1)}, \dots, E^{(p)}$ be vector subspaces of $(\mathbb{F}_p)^{p-1}$ not all equal to the null subspace.
    For every $j\in \{1,\dots, p\}$ 
    we denote with $B^{(j)}$ the set of all directed automorphisms along $P_j$ with defining vector in $E^{(j)}$.
    Then we say that $K = \langle a, B^{(1)}, \dots, B^{(p)} \rangle $ is the \textit{multi-EGS-group associated to the tuple} $E=(E^{(1)},\dots , E^{(p)})$.
\end{defn}

We observe that, with the notation of the above definition, for every $j\in \{1,\dots,p\}$ there is a natural isomorphism from $E^{(j)}$ to $B^{(j)}$, which makes $B^{(j)}$ into an elementary abelian $p$-subgroup of $\Aut \mathcal{T}$. Such isomorphism is the map sending a vector $\mathbf{e}\in E^{(j)}$ to the directed automorphism (along $P_j$) defined by $\mathbf{e}$.
Let us remark, in addition, that
\begin{equation}
\label{EQgenerateK}
    K=\langle a, \mathcal{B}^{(1)}, \dots , \mathcal{B}^{(p)} \rangle
\end{equation}
whenever $\mathcal{E}^{(j)}$ is a generating set for $E^{(j)}$ and $\mathcal{B}^{(j)}$ is the set of all directed automorphisms along $P_j$ with defining vector in $\mathcal{E}^{(j)}$.
Every multi-EGS-group is then finitely generated.


Some subfamilies of multi-EGS-groups have played an important role in modern group theory and deserve a different name.
\begin{defn}
    Let $K$ a multi-EGS-group as in Definition \ref{D1}.

    If $E^{(1)},\dots, E^{(p-1)}$ are all equal to the null subspace, we say that $K$ is the \textit{multi-GGS-group with defining space} $E^{(p)}$.

    If $E^{(1)},\dots, E^{(p-1)}$ are all equal to the null subspace and $E^{(p)}$ has dimension $1$, we say that $K$ is the \textit{GGS-group with defining space} $E^{(p)}$. In this case, taken any non zero vector $\mathbf{e}\in E^{(p)}$, $K$ is also said to be the \textit{GGS-group with defining vector} $\mathbf{e}$ (this is consistent thanks to (\ref{EQgenerateK})).
\end{defn}
The aforementioned Gupta-Sidki group is the GGS-group obtained by choosing $p\ge 3$ and $\mathbf{e}=(1,-1,0,\dots,0)$.
Theorem 1 of \cite{Infinite} ensures that every GGS-group is infinite. With the same argument, one can prove that $\langle a,b \rangle $ is infinite whenever $b$ is a directed automorphism with non-zero defining vector. It readily follows that every multi-EGS-group is infinite.

Another version of the problem raised by Burnside, known as the Burnside Problem, asks whether a finitely generated group with finite exponent has to be finite. The answer is still negative, but examples of infinite finitely generated groups with finite exponent cannot be found among the subgroups of $\Aut \mathcal{T}$.
To motivate
this, we look at
a third version of the Burnside Problem, which goes by the name of Restricted Burnside Problem and reads: 
given a finite group with exponent $m$ and $d$ generators, is it possible to find a bound for the order of such group which only depends on $m$ and $d$?
In the early 1990s, Zelmanov showed, in two remarkable papers (\cite{ZelmanovODD} and \cite{Zelmanov2}), that the answer to the above question is affirmative.
As a consequence, every finitely generated residually finite group with finite exponent has to be finite.
Hence every infinite finitely generated subgroup $G$ of $\Aut \mathcal{T}$ (and in particular every multi-EGS-group) must have infinite exponent and, by (\ref{EQlimiteG_n}), the sequence $\exp(G_n)$ goes to infinity when $n$ goes to infinity.
The aim of this paper will be to determine how fast $\exp(G_n)$ goes to infinity when $G$ is a multi-EGS-group.
This goal will be achieved in Theorem \ref{mainthm}, which is our main result.

Since any multi-EGS-group is infinite and finitely generated, the multi-EGS-groups giving a negative answer to the General Burnside Problem are exactly the periodic ones. 
Moreover, by Lemma 3.13 in \cite{EGS}, we have that a multi-EGS-group $K$ is periodic if and only if, with the notation of Definition \ref{D1}, the subspaces ${E^{(1)}},\dots , {E^{(p)}}$ are all contained in the hyperplane of $(\mathbb{F}_{p})^{p-1}$ defined by $V= \{(f_1,\dots,f_{p-1})\in (\mathbb{F}_{p})^{p-1}:f_1 + \dots + f_{p-1}=0\}$.
Given a multi-EGS-group $K$, Theorem \ref{mainthm} provides an explicit formula for $\exp(K_n)$, which only depends on the periodicity of $K$.
\begin{mainthm}
\label{mainthm}
    Let $K$ be a multi-EGS-group. Then for every $n\ge 1$
    \begin{equation*}
    \exp(K_n)=
    \begin{cases}
        p^{\lfloor \frac{n+1}{2} \rfloor } &\text{ if } K \text{ is periodic}\\
        p^n &\text{ otherwise}
    \end{cases}
\end{equation*}
where $\lfloor \frac{n+1}{2} \rfloor$ is the greatest integer number less than or equal to $\frac{n+1}{2}$.
\end{mainthm}

We observe that, given $K$ any multi-EGS-group, $\exp (K_n)$ can be trivially bounded from above.
Indeed, if we set
\begin{equation*}
    \Gamma = \{f\in \Aut\mathcal{T} : f_{(u)} \text{ is a power of } \sigma =(1 2 \dots p) \text{ for any vertex } u \}
\end{equation*}
we have $K \le \Gamma$ and, since $\Gamma_n$ is a Sylow $p$-subgroup of $\Sym(p^n)$,
\begin{equation}
\label{EQtrivialbound}
    \exp (K_n) \mid  \exp (\Gamma_n) = p^n.
\end{equation}
What is surprising and more delicate to be proved is that, starting from $\exp (K_1)=p$, the exponent of $K_n$ does grow by $p$ whenever $n$ increases by 1 (if $K$ is non-periodic) or whenever $n$ increases by 2 (if $K$ is periodic).
Therefore, not only, as the affirmative answer to the Restricted Burnside Problem guarantees, the increasing sequence $\exp (K_n)$ cannot converge to a finite number, but its growth to infinity is exponential in $n$ (which is, according to (\ref{EQtrivialbound}), the quickest growth we could expect).

\vspace{10 pt}

\noindent \textit{Notation.}
We will use the convention that $x^y=y^{-1}x y$ and $[x,y]=x^{-1}y^{-1}xy$.
Moreover, we will write $fg$ for the composition of two maps $f$ and $g$, where we apply first $f$ and then $g$ (i.e. $fg(u)=g(f(u))$).

\vspace{10 pt}

\noindent \textbf{Acknowledgements.}
I am deeply grateful to G.\ A.\ Fern\'{a}ndez-Alcober for the valuable suggestions and the continued support.
I thank the Basque Center for Applied Mathematics (BCAM) and the SEND Mobility Consortium for financial contributions to the carrying out of this work.

\section{Non-periodic multi-EGS-groups}

In this section we prove Theorem \ref{mainthm} for non-periodic multi-EGS-groups.
By (\ref{EQtrivialbound}), it is enough to show that $\exp (K_n) \ge p^n$ for any non-periodic multi-EGS-group $K$.

It can be easily checked that multi-EGS-groups are self-similar and that will be an important ingredient for the purpose of proving Theorem \ref{mainthm}.
We highlight some properties of self-similar subgroups of $\Aut \mathcal{T}$ which will be useful in this and in the following section.
\begin{lemma}
\label{L21}
    Let $S$ be a self-similar subgroup of $\Aut\mathcal{T}$. Then, for any natural numbers $k,n$ such that $ k\leq n$, the map
    \begin{align*}
    \hspace{0,5 cm} \St_{S_{n}}(k)\hspace{0,3 cm} &\longrightarrow S_{n-k}\times \overset{p^k}{\cdots} \times S_{n-k}
    \\
    \hspace{0,4 cm}g \St_{S}(n) \hspace{0,2 cm} &\longmapsto \hspace{0 cm} \big( g_u \St_{S}(n-k) \big)_{u \in X^{k}}
\end{align*}
is an injective group homomorphism.
\end{lemma}

We say that an element $g\in S$ has order $t$ in $S_n$ if $g \St_S(n)$ has order $t$ in $S_n$.
Then, provided that $S$ is self-similar, we have the following consequence of Lemma \ref{L21}:
\begin{equation}
\label{EQ1}
    \text{if a component of } \psi(g) \text{ has order } t \text{ in } S_n \implies g \text{ has order } \ge t \text{ in } S_{n+1}.
\end{equation}
Another consequence of Lemma \ref{L21} is stated in the following result.
\begin{lemma}
\label{L4}
    Let $S\leq \Aut\mathcal{T}$ be a self-similar group. Then
    \begin{equation*}
        \exp(S_{n+k}) \le \exp(S_n) \cdot \exp(S_k)
    \end{equation*}
    for every $n,k \geq 0$.
\end{lemma}
\begin{proof}
Let us observe that for any $n,k \geq 0$
\begin{equation*}
    S_{n} = \frac{S}{\St_{S}(n)} \simeq \frac{\frac{S}{\St_{S}(n+k)}}{\frac{\St_{S}(n)}{\St_{S}(n+k)}} = \frac{S_{n+k}}{\St_{S_{n+k}}(n)}
\end{equation*}
which yields that $\exp(S_{n+k}) $ divides $\exp(S_n) \cdot \exp(\St_{S_{n+k}}(n))$. Then it suffices to show that $\exp(\St_{S_{n+k}}(n))$ divides $\exp(S_k)$.
This follows from the fact that $\St_{S_{n+k}}(n)$ can be embedded in $S_k\times \overset{p^n}{\cdots} \times S_k$
by Lemma \ref{L21}. 
\end{proof}

We first prove Theorem \ref{mainthm} for non-periodic GGS-groups.
From now until the end of the paper, we will denote with $G$ the GGS-group with defining vector $\mathbf{e}=(e_1,\dots,e_{p-1})$ and we will indicate with $c$ the directed automorphism along $P_p$ defined by $\mathbf{e}$.
From what we observed in (\ref{EQgenerateK}), $G=\langle a,c \rangle $.
If we set $c_i:=c^{(a^i)}$ for any integer $i$, we have $c_i = c_j$ whenever $i\equiv j(\text{mod }p)$ and we get, by (\ref{EQconjugatebya}), that
\begin{align}
\label{Z1}
\begin{split}
    \psi(c_0) &= (a^{e_1},a^{e_2},\dots,a^{e_{p-1}},c)
    \\
    \psi(c_1)&=(c,a^{e_1},\dots,a^{e_{p-2}},a^{e_{p-1}})
    \\
    &\vdots
    \\
    \psi(c_{p-1}) &= (a^{e_2},a^{e_3},\dots,c,a^{e_1}).
    \end{split}
\end{align}
We will denote, in addition, by $\pi _i$ the projection $\pi_i : G \times \overset{p}{\dots} \times G \longrightarrow G$ on the $i$th component.

\begin{lemma}
\label{P}
     Assume $G$ to be a non-periodic GGS-group and let $x$ be the only integer in $\{0,\dots,p-1\}$ whose class in $\mathbb{F}_p$ is equal to $e_1+\dots+e_{p-1}$.
     Then, for every $n\geq 1$, any element of form $a^x c_j$ with $j\in \{0,\dots,p-1\}$ has order $\ge p^n$ in $G_n$.
\end{lemma}
\noindent \textit{Proof.}
We argue by induction on $n$.
For $n=1$, $G_1 \simeq \langle a \rangle 
$ is the cyclic group of order $p$. Since $G$ is not periodic, $x \not \equiv 0 (\text{mod }p) $ and then, for any $j\in \{0,\dots,p-1\}$, $a^x c_j \not \in \St_{G}(1)$, which yields that $a^x c_j$ has order $p$ in $G_1$.

Assume now $n\geq 2$ and let $j\in \{0,\dots,p-1\}$.
We have
\begin{align*}
    (a^x c_j)^p &= a^{xp} {c_j}^{(a^{x(p-1)})} {c_j}^{(a^{x(p-2)})} \dots {c_j}^{(a^{x})} c_j
    \\
    &=c_{j+x(p-1)} c_{j+x(p-2)} \dots c_{j+x} c_j
\end{align*}
and, since $x \not \equiv 0 (\text{mod }p)$, $\{j+x(p-1),j+x(p-2),\dots,j+x,j\}$ is a set of representatives for the congruence classes modulo $p$. In other words, $(a^x c_j)^p $ is the product, in some order, of the 
$p$ factors $c_0, \dots, c_{p-1}$. Then,
by (\ref{Z1}), the first component of $\psi((a^x c_j)^p) $ is the product, in some order, of the $p$ factors $a^{e_1}, \dots, a^{e_{p-1}},c$. This means that 
\begin{align*}
     \psi\pi_1((a^x c_j)^p) &= a^{f_1}\dots a^{f_k} c a^{f_{k+1}} \dots a^{f_{p-1}}
    \\
    &= a^{f_1 + \dots + f_{p-1}} c^{(a^{f_{k+1}+\dots +f_{p-1}})}
\end{align*}
for some $0\le k\le p-1$ and $f_1+\dots+ f_{p-1}=e_1+ \dots + e_{p-1}$.
Then $ \psi\pi_1((a^x c_j)^p)$ is of the form $a^x c_s$ for some $s\in \{0,\dots,p-1\}$ and by induction hypothesis it has order $\ge p^{n-1}$ in $G_{n-1}$. This yields, by (\ref{EQ1}), that $a^x c_j$ has order $\geq p^n$ in $G_n$. \qed

\vspace{10 pt}

Now the following result is straightforward.
\begin{prop}
\label{CC}
    Let $G$ be a non-periodic GGS-group. Then
    \begin{equation*}
        \exp(G_n) = p^n.
    \end{equation*}
    for every $n\ge 1$.
\end{prop}

To extend Proposition \ref{CC} to any non-periodic multi-EGS-group, it will suffice to realise that a non-periodic multi-EGS-group always contains a conjugate of a non-periodic GGS-group.

\begin{theorem}
\label{THMmain_nonperiodic}
    Let $K$ be a non-periodic multi-EGS-group. Then
    \begin{equation*}
        \exp (K_n) = p^n
    \end{equation*}
    for every $n\ge 1$.
\end{theorem}
\noindent \textit{Proof.}
Take $K$ a multi-EGS-group as in Definition \ref{D1}. For every $j\in \{1,\dots,p\}$ and every $\mathbf{e}\in E^{(j)}$, we set $H_{\mathbf{e}}^{(j)}=\langle a,b \rangle$, where $b$ is the directed automorphism along $P_j$ with defining vector $\mathbf{e}$.

We first show that, fixed $j\in \{1,\dots,p\}$ and $\mathbf{e}=(e_1, \dots, e_{p-1})\in E^{(j)}\setminus \{\mathbf{0}\}$, $H_{\mathbf{e}}^{(j)}=\langle a,b \rangle$ is conjugate to the GGS-group $L=\langle a,d \rangle $, where $d$ is the directed automorphism along $P_p$ with defining vector $\mathbf{e}$.
We denote with $\sigma$ the $p$-cycle $(1 2 \dots p) \in \mathrm{Sym}(p)$ and
we call $f$ the automorphism of $\mathcal{T}$ that has all the labels equal to $\sigma ^{p-j}$. In other words, $f=x a^{p-j}$, with $x\in \St(1)$ defined by $\psi(x)=(f,\dots , f)$.
By (\ref{EQconjugatebya}) we have $\psi(x^a)=(f,\dots,f)=\psi(x)$, then $x$ commutes with $a$ and $f$ commutes with $a$.
It follows from this and again from (\ref{EQconjugatebya}) that
\begin{equation*}
    \psi(b^x)= (a^{e_{p-j+1}},\dots,a^{e_{p-1}}, b^f, a^{e_1},\dots,a^{e_{p-j}}) 
\end{equation*}
and 
\begin{equation*}
    \psi(b^f)=(a^{e_1}, \dots, a^{e_{p-1}}, b^f),
\end{equation*}
which yields that $b^f=d$. All in all, we have $(H_{\mathbf{e}}^{(j)})^f=L$, as desired.

Now, $K$ is not periodic and then there is $\mathbf{e}$ which lies in some ${E^{(j)}} $ but not in $V$. For what we just proved, $H_{\mathbf{e}}^{(j)}\le K$ is conjugate to a non-periodic GGS-group. Thus the $n$th congruence quotient of $H_{\mathbf{e}}^{(j)}$ is isomorphic to the $n$th congruence quotient of such GGS-group, and it can be embedded in $K_n$. It follows that $\exp (K_n)\ge p^n$ by Proposition \ref{CC}, and we have the claim thanks to (\ref{EQtrivialbound}). \qed


\section{Periodic multi-EGS-groups}
In this section we prove Theorem \ref{mainthm} for a periodic multi-EGS-group $K$.
As for the non-periodic case, it will be easy to show that the expression provided by Theorem \ref{mainthm} is an upper bound for $\exp (K_n)$. More efforts will be needed to prove that such expression is a lower bound too.

We first compute the exponent of $K_2$.
For such purpose, we need to fix the following notation: for every $W\subseteq (\mathbb{F}_{p})^{p-1}$, we indicate with $\overline{W}$ the subset of $(\mathbb{F}_{p})^{p}$ given by
\begin{equation*}
    \overline{W}=\{(e_{p-j+1},\dots,e_{p-1},0,e_1,\dots,e_{p-j}): 1\le j\le p \text{ and } (e_1,\dots,e_{p-1})\in W \}.
\end{equation*}
\begin{lemma}
\label{L11}
   Let $K$ be a multi-EGS-group as in Definition \ref{D1}. Then
   \begin{equation*}
       |K_2|= p^{t+1},
   \end{equation*}
   where $t$ is the dimension of the $\mathbb{F}_p$-vector space $\langle 
   \overline{E^{(j)}} : j=1,\dots,p \rangle$.
\end{lemma}
\noindent \textit{Proof.}
Observe that $K= \langle a \rangle \ltimes \St_K(1)$ and $K_2
$ has order equal to $|\langle a \rangle| \cdot  |\frac{\St_K(1)}{\St_K(2)}| = p \cdot  |\frac{\St_K(1)}{\St_K(2)}|$. Then it suffices to show that $|\St_{K_{2}}(1)| = p^t$, with $t$ as in the statement. Now, $K_1
\simeq \langle a \rangle \simeq  \mathbb{F}_p$ and the composition $\varphi$ between the map
\begin{align*}
   \hspace{0,5 cm}\hspace{0,2 cm} \St_{K_{2}}(1)\hspace{0,3 cm} &\longrightarrow K_1\times \overset{p}{\cdots} \times K_1
   \\
   \hspace{0,4 cm}g \St_{K}(2) \hspace{0,2 cm} &\longmapsto (g_1 \St_K(1),\dots,g_p \St_K(1))
\end{align*}
and the map
\begin{align*}
     \hspace{1,7 cm} K_1\times \overset{p}{\cdots} \times K_1 \hspace{1,5 cm} &\longrightarrow (\mathbb{F}_p)^p
    \\
    (a^{f_1}\St_K(1),\dots , a^{f_p}\St_K(1)) \hspace{0,2 cm}&\longmapsto (f_1,\dots,f_p),
\end{align*}
is, by Lemma \ref{L21}, an injective homomorphism.
It can be easily proved, by using (\ref{EQconjugatebya}), that $\bigcup_{j=1}^{p} \overline{E^{(j)}} = \varphi(\Delta)$ where $\Delta= \{{b}^{(a^l)}: b\in \bigcup_{j=1}^{p} B^{(j)} \text{ and } 0\le l\le p-1 \}$.


If we prove that $\St_K(1)$ is generated by $\Delta$, we will get that $\varphi (\St_{K_{2}}(1))$ is the vector subspace of $(\mathbb{F}_p)^p$ generated by $\bigcup_{j=1}^{p} \overline{E^{(j)}}$, and this will prove the claim since $|\varphi (\St_{K_{2}}(1))|= |\St_{K_{2}}(1)|$.

Then we are only left to show that $\St_K(1)$ is generated by $\Delta$. This follows from the fact that $\langle \Delta \rangle \leq \St_K(1)$, $\langle \Delta \rangle $ is normal in $K$ and the quotients $\frac{K}{\langle \Delta \rangle }$, $\frac{K}{\St_K(1)}$ have the same order (equal to $p$) since they are both generated by the class of $a$. \qed
\vspace{5 pt}
\begin{lemma}
\label{L22}
    Let $K$ be a periodic multi-EGS-group. Then the exponent of $K_2$ is $p$.
\end{lemma}
\noindent \textit{Proof.}
Since $K$ is periodic, keeping as usual the notation of Definition \ref{D1}, we have that $\bigcup_{j=1}^{p} \overline{E^{(j)}}$ is contained in $U= \{(f_1,\dots,f_{p})\in (\mathbb{F}_{p})^{p}:f_1 + \dots + f_{p}=0\}$.
Since $U$ is a vector space of dimension $p-1$, Lemma \ref{L11} yields that $|K_2|\leq p^p$. This implies, by Theorem 2.8 of \cite{Reg}, that $K_2$ is a regular $p$-group and, being generated by elements of order $p$, it must have exponent $p$ by Corollary 2.11 of \cite{Reg}. \qed

\vspace{10 pt}

Since $\exp (K_2) =p$ and $K$ is self-similar, Lemma \ref{L4} yields that $\exp (K_n)$ can grow at most by $p$ whenever $n$ increases by 2. This allows us to bound from above $\exp (K_n)$.
\begin{prop}
\label{C10}
    Let $K$ be a periodic multi-EGS-group. Then
    \begin{equation*}
        \exp(K_n)\leq p^{\lfloor \frac{n+1}{2} \rfloor}
    \end{equation*}
    for every $n\ge 1$.
\end{prop}


Analogously to the non-periodic case, our strategy to bound $\exp (K_n)$ from below will be to do that when $K$ is a GGS-group and then to look for a suitable GGS-group inside any multi-EGS-group.
The most noteworthy part of our proof comes with the next lemma.
\begin{lemma}
    \label{L5}
    Assume $G$ to be a periodic GGS-group and let $n\ge 1$ be an integer.
    Then
    \begin{equation*}
        \exp(G_{n+2}) \geq \exp(G_{n}) \cdot p
    \end{equation*}
\end{lemma}

\noindent \textit{Proof.}
We deduce from (\ref{EQtrivialbound}) that $\exp(G_n) = p^k$ for some $k\geq 1$; then there is $g\in G$ which has order $p^k$ in $G_n$. In order to prove the claim, it is enough to find an element of $G$ that has order $ \geq p^{k+1}$ in $G_{n+2}$.

By Proposition 2.2 in \cite{AG}, the composition $\psi \pi_i : \St_{G}(1) \longrightarrow G$ is a surjective homomorphism.
Hence there is $h\in \St_{G}(1)$ such that $\psi \pi_1 (h) = g$ and, by (\ref{EQ1}), the order of $h$ in $G_{n+1}$ is $\geq p^k$.

Since $h \in \St_{G}(1)= \langle c_0,\dots,c_{p-1} \rangle$ (see the last part of the proof of Lemma \ref{L11}) and $c_0,\dots,c_{p-1}$ have order $p$, we can write
\begin{equation*}
    h= c_{j_1}\dots c_{j_{r}}
\end{equation*}
for some $r\geq 1$ and $j_1,\dots,j_{r} \in \{0,\dots,p-1 \}$. By Theorem 2.16 in \cite{AG}, we may assume $e_1=1$, up to replacing $G$ with a GGS-group that is conjugate to $G$ (obviously the periodicity of the group and the exponents of congruence quotients do not change if we replace $G$ with a conjugate of $G$). Then we define the element
\begin{equation*}
    f_1={c_1}^{({c}^{j_1})} {c_1}^{({c}^{j_2})}\dots {c_1}^{({c}^{j_r})}.
\end{equation*}
Observe that for any integer $\lambda$
\begin{align*}
    \psi({c_1}^{({c}^{\lambda})})&=
    (c,a,\dots,a^{e_{p-2}},a^{e_{p-1}})^{(a^{\lambda},a^{\lambda e_2},\dots,a^{\lambda e_{p-1}},c^{\lambda})}
    \\
    &= (c_{\lambda},a,\dots,a^{e_{p-2}},{(a^{e_{p-1}})}^{(c^{\lambda})})
\end{align*}
and then
\begin{equation*}
    \psi(f_1)=(h,a^r,\dots,a^{r e_{p-2}},F)
\end{equation*}
where
\begin{align*}
    F&= (a^{e_{p-1}})^{(c^{j_1})}\dots (a^{e_{p-1}})^{(c^{j_r})}
    \\
    &=a^{e_{p-1}} [a^{e_{p-1}},c^{j_1}]\dots a^{e_{p-1}} [a^{e_{p-1}},c^{j_r}]
    \\
    &= a^{r e_{p-1}} [a^{e_{p-1}},c^{j_1}]^{(a^{(r-1)e_{p-1}})} [a^{e_{p-1}},c^{j_2}]^{(a^{(r-2)e_{p-1}})} \dots [a^{e_{p-1}},c^{j_r}].
\end{align*}
We set $y=[a^{e_{p-1}},c^{j_1}]^{(a^{(r-1)e_{p-1}})} [a^{e_{p-1}},c^{j_2}]^{(a^{(r-2)e_{p-1}})} \dots [a^{e_{p-1}},c^{j_r}]$, so that $F=a^{r e_{p-1}} y$.

At this point, it is convenient to consider the case in which $\mathbf{e}$ is non-symmetric and the case in which $\mathbf{e}$ is symmetric separately. In both cases $\mathbf{e}\in (\mathbb{F}_p)^{p-1} \setminus \{\mathbf{0}\}$ cannot be constant (because the sum of its components is 0 by the periodicity of $G$) and in particular $p \neq 2$.

Assume first $\mathbf{e}$ non-symmetric.
In this case, Lemma 3.4 of \cite{AG} yields that $\gamma_2(G) \times \overset{p}{\dots} \times \gamma_2(G) \subseteq \psi(\St_G(1))$.
Then there exists $f_2\in \St_G(1)$ such that $\psi(f_2)=(1,\dots,1,y^{-1})$ and we get 
\begin{equation*}
    \psi(f_1 f_2) = (h,a^r,a^{r e_2},\dots,a^{r e_{p-2}},a^{r e_{p-1}}).
\end{equation*}
Since $G$ is periodic we have $e_1+e_2 +\dots+e_{p-1} = 0$ and (\ref{EQpth_powers}) yields that
\begin{equation*}
    \psi \pi_{p} ((a f_1 f_2)^p) = h \cdot a^{r(e_1+e_2+\dots+e_{p-1})}=h.
\end{equation*}
This implies, by (\ref{EQ1}), that $a f_1 f_2$ has order $\geq p^{k+1}$ in $G_{n+2}$, as desired.

Let us now consider the case in which $\mathbf{e}$ is symmetric. 
Since we are assuming $e_1=1$, we have $e_{p-1}=1$.

Recall that we set $h=c_{j_1}\dots c_{j_r}$ and the only assumption we made on $h$ is $\psi\pi_1(h)=g$.
We want to show that, thanks to the symmetry of $\mathbf{e}$, we may assume $r\equiv j_1+\dots+j_r \text{ (mod }p)$.
Since $2$ is invertible in $\mathbb{F}_p$ (remember that $p\neq 2$), there is an integer $\alpha$ such that $2 \alpha \equiv r - (j_1+\dots + j_r) \text{ (mod }p)$. Then the element 
\begin{equation*}
    h' = ({c_0}^{p-1} c_2)^{\alpha} h 
\end{equation*}
can be written in the form $h' = c_{l_1}\dots c_{l_s}$, with $s = p\alpha +r \equiv r $ and $l_1+\dots+l_s \equiv 2\alpha + j_1 +\dots+j_r \equiv  r \text{ (mod }p)$. Moreover, remembering that $e_1=e_{p-1}=1$ and using (\ref{Z1}), we have $\psi \pi_1 (h') = (a^{p-1}a)^{\alpha} \psi \pi_1 (h)=g$. Then $h'= c_{l_1}\dots c_{l_s}$ satisfies $s\equiv l_1+\dots+l_s \text{ (mod }p)$ and $\psi \pi_1 (h') = g$.
This allows us to suppose, from now on, that $r\equiv j_1+\dots+j_r \text{ (mod }p)$.

We have
\begin{align*}
    y &= [a,c^{j_1}]^{(a^{r-1})} [a,c^{j_2}]^{(a^{r-2})} \dots [a,c^{j_r}] 
    \\
    &\equiv [a,c^{j_1}] [a,c^{j_2}] \dots [a,c^{j_r}]
    \\
    &\equiv [a,c]^{j_1+\dots+j_r}
    \\
    &\equiv [a,c]^r \text{ (mod } \gamma_3(G) \text{)},
\end{align*}
where the last congruence holds because $[a,c]$ has order $p$ in the quotient $\gamma_2(G) / \gamma_3(G)$ (see Theorem 2.1 in \cite{AG}).
This yields $y=z \cdot [a,c]^r$ for some $z \in \gamma_3(G)$. Working again modulo $\gamma_3(G)$, we get 
\begin{align*}
    [a,h]^{(a^{-1})} &\equiv [a,c_{j_1}\dots c_{j_r}]
    \\
    &\equiv [a,c]^{(a^{j_1})}\dots [a,c]^{(a^{j_r})}
    \\
    &\equiv [a,c]^r\text{ (mod } \gamma_3(G) \text{)}
\end{align*}
and $[a,c]^r=[a,h]^{(a^{-1})}\cdot w$ for some $w\in \gamma_3(G)$.
Now, since $\mathbf{e}$ is non-constant, Lemma 3.2 of \cite{AG} ensures that $\gamma_3(G) \times \overset{p}{\dots} \times \gamma_3(G) \subseteq \psi(\St_G(1))$.
Then there must be $f_3\in \St_{G}(1)$ with $\psi(f_3) = (w^{-1},1,\dots,1,z^{-1})$ and
\begin{align*}
    \psi (f_1 [c_0,c_1]^r f_3)&= (h, a^{r },\dots,a^{r e_{p-2}},a^{r } \cdot y)([a,c]^r,1,\dots,1,[c,a]^r)(w^{-1},1,\dots,1,z^{-1})
    \\
    &= (h\cdot [a,h]^{(a^{-1})}, a^{r },\dots,a^{r e_{p-2}},a^r)
    \\
    &= (h^{(a^{-1})}, a^{r e_1},\dots,a^{r e_{p-2}},a^{r e_{p-1}}),
\end{align*}
where the last equality holds by the commutator identity $[a,h]^{(a^{-1})}= [h,a^{-1}]$.
This implies, analogously to the previous case, that
\begin{equation*}
    \psi \pi_{p} ((a f_1 [c_0,c_1]^r f_3)^p) = h^{(a^{-1})}
\end{equation*}
and that $a f_1 [c_0,c_1]^r f_3$ has order $\geq p^{k+1}$ in $G_{n+2}$. \qed

\vspace{10 pt}

Now the next proposition follows immediately by induction on $n$.
\begin{prop}
\label{C12}
Let $G$ be a periodic GGS-group. Then
    \begin{equation*}
        \exp(G_n) = p^{\lfloor \frac{n+1}{2} \rfloor} 
    \end{equation*}
    for every $n\ge 1$.
\end{prop}

To prove Theorem \ref{mainthm}, we are only left to extend Proposition \ref{C12} to any periodic multi-EGS-group.

\begin{theorem}
    Let $K$ be a periodic multi-EGS-group. Then
    \begin{equation*}
        \exp (K_n) = p^{\lfloor \frac{n+1}{2} \rfloor}
    \end{equation*}
    for every $n\ge 1$.
\end{theorem}
\noindent \textit{Proof.}
The proof follows the same line of the proof of Theorem \ref{THMmain_nonperiodic}. The only difference is that now $K$ is periodic, hence $  E^{(j)} \subseteq  V$ for every $j\in \{1,\dots ,p\}$ and, since $E^{(1)},\dots ,E^{(p)} $ are not all equal to the null subspace, we can pick $\mathbf{e}\in \bigcup_{j=1}^{p} E^{(j)} \setminus \{\mathbf{0}\} \subseteq V$. It follows that $H_{\mathbf{e}}^{(j)}$ is conjugate to a periodic GGS-group and the claim follows by Proposition \ref{C12} and Proposition \ref{C10}. \qed



\end{document}